\documentclass[11pt]{article}
\usepackage{amsmath,amsthm,amsfonts,amssymb,latexsym}
\usepackage{hyperref}
\usepackage{comment}
\usepackage{enumerate}
\usepackage{color}
\usepackage[dvipsnames]{xcolor}

\usepackage[shortlabels]{enumitem}

\usepackage{xparse,etoolbox}

\usepackage{sectsty}
\allsectionsfont{\centering}

\usepackage{parskip}
\makeatletter
\def\thm@space@setup{%
  \thm@preskip=\parskip \thm@postskip=0pt
}
\makeatother

\usepackage[margin=1.2in]{geometry}

\begin{document}

\theoremstyle{plain}

\newtheorem{thm}{Theorem}[section]

\newtheorem{lem}[thm]{Lemma}
\newtheorem{Problem B}[thm]{Problem B}

\newtheorem{pro}[thm]{Proposition}
\newtheorem{conj}[thm]{Conjecture}
\newtheorem{cor}[thm]{Corollary}
\newtheorem{que}[thm]{Question}

\theoremstyle{definition}
\newtheorem{rem}[thm]{Remark}
\newtheorem{defi}[thm]{Definition}
\newtheorem{hyp}[thm]{Hypothesis}

\theoremstyle{plain}
\newtheorem*{thmA}{Theorem A}
\newtheorem*{thmB}{Theorem B}
\newtheorem*{corB}{Corollary B}
\newtheorem*{thmC}{Theorem C}
\newtheorem*{thmD}{Theorem D}
\newtheorem*{thmE}{Theorem E}
 
\newtheorem*{thmAcl}{Main Theorem$^{*}$}
\newtheorem*{thmBcl}{Theorem B$^{*}$}
\newcommand{\dd}{\mathrm{d}}

\theoremstyle{plain}
\newtheorem{theoA}{Theorem}

\theoremstyle{plain}
\newtheorem{conjA}[theoA]{Conjecture}

\theoremstyle{plain}
\newtheorem{condA}[theoA]{Condition}

\theoremstyle{plain}
\newtheorem{paraA}[theoA]{Parametrisation}

\theoremstyle{plain}
\newtheorem{corA}[theoA]{Corollary}

\renewcommand{\thetheoA}{\Alph{theoA}}

\renewcommand{\thecorA}{\Alph{corA}}

\newcommand{\wh}[1]{\widehat{#1}} 
\newcommand{\miquelcomment}{\textcolor{blue}}
\newcommand{\ncomment}{\textcolor{magenta}}

\newcommand{\Maxn}{\operatorname{Max_{\textbf{N}}}}
\newcommand{\Syl}{\operatorname{Syl}}
\newcommand{\Lin}{\operatorname{Lin}}
\newcommand{\U}{\mathbf{U}}
\newcommand{\nav}{\mathrm{Nav}}
\newcommand{\R}{\mathbf{R}}
\newcommand{\dl}{\operatorname{dl}}
\newcommand{\Con}{\operatorname{Con}}
\newcommand{\rdz}{\operatorname{rdz}}
\newcommand{\rdzo}{\operatorname{rdz}^{\circ}}
\newcommand{\cl}{\operatorname{cl}}
\newcommand{\Stab}{\operatorname{Stab}}
\newcommand{\Aut}{\operatorname{Aut}}
\newcommand{\Ker}{\operatorname{Ker}}
\newcommand{\InnDiag}{\operatorname{InnDiag}}
\newcommand{\fl}{\operatorname{fl}}
\newcommand{\irr}{\operatorname{Irr}}
\newcommand{\ibr}{\operatorname{IBr}}
\newcommand{\FF}{\mathbb{F}}
\newcommand{\CL}{\mathfrak{Cl}}
\newcommand{\EE}{\mathbb{E}}
\newcommand{\Alp}{\mathrm{Alp}}
\newcommand{\Alpr}{\mathrm{Alp_r}}
\newcommand{\normal}{\unhld}
\newcommand{\sn}{\normal\normal}
\newcommand{\Bl}{\mathrm{Bl}}
\newcommand{\NN}{\mathbb{N}}

\newcommand{\N}{\mathbf{N}}
\newcommand{\bfC}{\mathbf{C}}
\newcommand{\bfO}{\mathbf{O}}
\newcommand{\bfF}{\mathbf{F}}
\newcommand{\Irr}{\mathrm{Irr}}
\def\GGG{{\mathcal G}}
\def\HHH{{\mathcal H}}
\def\HH{{\mathcal H}}
\def\irra#1#2{{\rm Irr}_{#1}(#2)}

\renewcommand{\labelenumi}{\upshape (\roman{enumi})}

\newcommand{\PSL}{\operatorname{PSL}}
\newcommand{\PSU}{\operatorname{PSU}}
\newcommand{\alt}{\operatorname{Alt}}

\providecommand{\V}{\mathrm{V}}
\providecommand{\E}{\mathrm{E}}
\providecommand{\ir}{\mathrm{Irm_{rv}}}
\providecommand{\Irrr}{\mathrm{Irm_{rv}}}
\providecommand{\re}{\mathrm{Re}}

\numberwithin{equation}{section}
\def\irrp#1{{\rm Irr}_{p'}(#1)}

\def\ibrrp#1{{\rm IBr}_{\Bbb R, p'}(#1)}
\def\C{{\mathbb C}}

\def\isoc{{\succeq_c}}

\def\isob{{\succeq_b}}

\newcommand{\wt}[1]{\widetilde{#1}} 

\def\Rad{{\rm Rad}}
\def\Rado{{\rm Rad}^{\circ}}

\def\o{{\bf O}}
\def\c{{\bf C}}
\def\n{{\bf N}}
\def\z{{\bf Z}}
\def\F{{\bf F}}
\def\P{{\mathcal{P}}}
\def\Q{{\mathcal{Q}}}
\def\R{{\mathcal{R}}}
\def\W{{\mathcal{W}}}
\def\D{{\mathcal{D}}}
\def\Wr{{\mathcal{W}_{\rm r}}}

\def\irr#1{{\rm Irr}(#1)}
\def\irrp#1{{\rm Irr}_{p^\prime}(#1)}
\def\irrq#1{{\rm Irr}_{q^\prime}(#1)}
\def \aut#1{{\rm Aut}(#1)}
\def\cent#1#2{{\bf C}_{#1}(#2)}
\def\norm#1#2{{\bf N}_{#1}(#2)}
\def\zent#1{{\bf Z}(#1)}
\def\syl#1#2{{\rm Syl}_#1(#2)}
\def\normal{\unlhd}
\def\oh#1#2{{\bf O}_{#1}(#2)}
\def\Oh#1#2{{\bf O}^{#1}(#2)}
\def\det#1{{\rm det}(#1)}
\def\gal#1{{\rm Gal}(#1)}
\def\ker#1{{\rm ker}(#1)}
\def\normalm#1#2{{\bf N}_{#1}(#2)}
\def\alt#1{{\rm Alt}(#1)}
\def\iitem#1{\goodbreak\par\noindent{\bf #1}}
   \def \mod#1{\, {\rm mod} \, #1 \, }
\def\sbs{\subseteq}

\def\gc{{\bf GC}}
\def\ngc{{non-{\bf GC}}}
\def\ngcs{{non-{\bf GC}$^*$}}
\newcommand{\notd}{{\!\not{|}}}

\newcommand{\Z}{\mathbf{Z}}

\newcommand{\bG}{\mathbf{G}}
\newcommand{\bL}{\mathbf{L}}
\newcommand{\bH}{\mathbf{H}}
\newcommand{\bM}{\mathbf{M}}

\newcommand{\ty}[1]{\mathsf{#1}}

\newcommand{\cE}{\mathscr{E}}

\newcommand{\Out}{{\mathrm {Out}}}
\newcommand{\Mult}{{\mathrm {Mult}}}
\newcommand{\Inn}{{\mathrm {Inn}}}
\newcommand{\Fong}{{\mathrm{Fong}}}
\newcommand{\IBRL}{{\mathrm {IBr}}_{\ell}}
\newcommand{\IBRP}{{\mathrm {IBr}}_{p}}
\newcommand{\bl}{{\mathrm{bl}}}
\newcommand{\cd}{\mathrm{cd}}
\newcommand{\ord}{{\mathrm {ord}}}
\def\id{\mathop{\mathrm{ id}}\nolimits}
\renewcommand{\Im}{{\mathrm {Im}}}
\newcommand{\Ind}{{\mathrm {Ind}}}
\newcommand{\diag}{{\mathrm {diag}}}
\newcommand{\soc}{{\mathrm {soc}}}
\newcommand{\End}{{\mathrm {End}}}
\newcommand{\sol}{{\mathrm {sol}}}
\newcommand{\Hom}{{\mathrm {Hom}}}
\newcommand{\Mor}{{\mathrm {Mor}}}
\newcommand{\Mat}{{\mathrm {Mat}}}
\def\rank{\mathop{\mathrm{ rank}}\nolimits}
\newcommand{\Tr}{{\mathrm {Tr}}}
\newcommand{\tr}{{\mathrm {tr}}}
\newcommand{\Gal}{{\rm Gal}}
\newcommand{\Spec}{{\mathrm {Spec}}}
\newcommand{\ad}{{\mathrm {ad}}}
\newcommand{\Sym}{{\mathrm {Sym}}}
\newcommand{\Char}{{\mathrm {Char}}}
\newcommand{\pr}{{\mathrm {pr}}}
\newcommand{\rad}{{\mathrm {rad}}}
\newcommand{\abel}{{\mathrm {abel}}}
\newcommand{\PGL}{{\mathrm {PGL}}}
\newcommand{\PCSp}{{\mathrm {PCSp}}}
\newcommand{\PGU}{{\mathrm {PGU}}}
\newcommand{\codim}{{\mathrm {codim}}}
\newcommand{\ind}{{\mathrm {ind}}}
\newcommand{\Res}{{\mathrm {Res}}}
\newcommand{\Lie}{{\mathrm {Lie}}}
\newcommand{\Ext}{{\mathrm {Ext}}}
\newcommand{\EBr}{{\mathrm {EBr}}}
\newcommand{\Alt}{{\mathrm {Alt}}}
\newcommand{\AAA}{{\sf A}}
\newcommand{\SSS}{{\sf S}}
\newcommand{\DDD}{{\sf D}}
\newcommand{\QQQ}{{\sf Q}}
\newcommand{\CCC}{{\sf C}}
\newcommand{\SL}{{\mathrm {SL}}}
\newcommand{\Sp}{{\mathrm {Sp}}}
\newcommand{\PSp}{{\mathrm {PSp}}}
\newcommand{\SU}{{\mathrm {SU}}}
\newcommand{\GL}{{\mathrm {GL}}}
\newcommand{\GU}{{\mathrm {GU}}}
\newcommand{\Br}{{\mathrm{Br}}}
\newcommand{\Spin}{{\mathrm {Spin}}}
\newcommand{\CC}{{\mathbb C}}
\newcommand{\CB}{{\mathbf C}}
\newcommand{\RR}{{\mathbb R}}
\newcommand{\QQ}{{\mathbb Q}}
\newcommand{\ZZ}{{\mathbb Z}}
\newcommand{\bfN}{{\mathbf N}}
\newcommand{\bfZ}{{\mathbf Z}}
\newcommand{\PP}{{\mathbb P}}
\newcommand{\cG}{{\mathcal G}}
\newcommand{\cH}{{\mathcal H}}
\newcommand{\cQ}{{\mathcal Q}}
\newcommand{\GA}{{\mathfrak G}}
\newcommand{\cT}{{\mathcal T}}
\newcommand{\cL}{{\mathcal L}}
\newcommand{\IBr}{\mathrm{IBr}}
\newcommand{\cS}{{\mathcal S}}
\newcommand{\cR}{{\mathcal R}}
\newcommand{\GCD}{\GC^{*}}
\newcommand{\TCD}{\TC^{*}}
\newcommand{\FD}{F^{*}}
\newcommand{\GD}{G^{*}}
\newcommand{\HD}{H^{*}}
\newcommand{\GCF}{\GC^{F}}
\newcommand{\TCF}{\TC^{F}}
\newcommand{\PCF}{\PC^{F}}
\newcommand{\GCDF}{(\GC^{*})^{F^{*}}}
\newcommand{\RGTT}{R^{\GC}_{\TC}(\theta)}
\newcommand{\RGTA}{R^{\GC}_{\TC}(1)}
\newcommand{\Om}{\Omega}
\newcommand{\eps}{\epsilon}
\newcommand{\varep}{\varepsilon}
\newcommand{\dz}{\mathrm{dz}}
\newcommand{\dzo}{\mathrm{dz}^\circ}
\newcommand{\Co}{\mathcal{C}^\circ}
\newcommand{\al}{\alpha}

\newcommand{\chis}{\chi_{s}}
\newcommand{\sigmad}{\sigma^{*}}
\newcommand{\PA}{\boldsymbol{\alpha}}
\newcommand{\gam}{\gamma}
\newcommand{\lam}{\lambda}
\newcommand{\la}{\langle}
\newcommand{\genf}{F^*}
\newcommand{\ra}{\rangle}
\newcommand{\hs}{\hat{s}}
\newcommand{\htt}{\hat{t}}
\newcommand{\tG}{\hat G}
\newcommand{\St}{\mathsf {St}}
\newcommand{\bfs}{\boldsymbol{s}}
\newcommand{\bfl}{\boldsymbol{\lambda}}
\newcommand{\tn}{\hspace{0.5mm}^{t}\hspace*{-0.2mm}}
\newcommand{\ta}{\hspace{0.5mm}^{2}\hspace*{-0.2mm}}
\newcommand{\tb}{\hspace{0.5mm}^{3}\hspace*{-0.2mm}}
\def\skipa{\vspace{-1.5mm} & \vspace{-1.5mm} & \vspace{-1.5mm}\\}
\newcommand{\tw}[1]{{}^#1\!}
\renewcommand{\mod}{\bmod \,}

\let\ti=\times
\let\la=\lambda
\let\eps=\epsilon

\marginparsep-0.5cm

\newcommand{\blocktheorem}[1]{%
  \csletcs{old#1}{#1}% Store \begin
  \csletcs{endold#1}{end#1}% Store \end
  \RenewDocumentEnvironment{#1}{o}
    {\par\addvspace{1.5ex}
     \noindent\begin{minipage}{\textwidth}
     \IfNoValueTF{##1}
       {\csuse{old#1}}
       {\csuse{old#1}[##1]}}
    {\csuse{endold#1}
     \end{minipage}
     \par\addvspace{1.5ex}}
}

\blocktheorem{theoA}
\blocktheorem{conjA}

\makeatletter
\def\blfootnote{\gdef\@thefnmark{}\@footnotetext}
\makeatother

\title{{\bf{\huge Alperin's bound and normal Sylow subgroups}}}

\author{Zhicheng Feng, J. Miquel Mart\'inez, and Damiano Rossi}

\blfootnote{\emph{$2020$ Mathematical Subject Classification:} $20$C$20$ ($20$C$15$, $20$C$33$)
\\
The first-named author gratefully acknowledges financial support by NSFC (No.12350710787, 12431001) and Startup Foundation of Shenzhen (Y01916102).
The work of the second-named author is funded by the European Union -- Next Generation EU, Missione 4 Componente 1, PRIN 2022-2022PSTWLB -- Group Theory and Applications, CUP B53D23009410006 as well as the Spanish Ministerio de Ciencia e Innovaci\'on (Grant PID2022-137612NB-I00 funded by MCIN/AEI/ 10.13039/501100011033 and “ERDF A way of making Europe”). The third-named author is supported by the Walter Benjamin Programme of the DFG - Project number 525464727.
\\
\date{}
}

\maketitle

\begin{abstract}
Let $G$ be a finite group, $p$ a prime number and $P$ a Sylow $p$-subgroup of $G$. Recently, G. Malle, G. Navarro, and P. H. Tiep conjectured that the number of $p$-Brauer characters of $G$ coincides with that of the normaliser $\n_G(P)$ if and only if $P$ is normal in $G$. We reduce this conjecture to a question about finite simple groups and prove it for the prime $p=2$. As a by-product of our work, we prove a reduction theorem for the blockwise version of Alperin's lower bound on $p$-Brauer characters and prove it for $2$-blocks of maximal defect. This improves recent results obtained by Malle, Navarro, and Tiep.
\end{abstract}

\section{Introduction}

Let $G$ be a finite group, $p$ a prime number, and $P$ a Sylow $p$-subgroup of $G$. According to the It\^o--Michler theorem $p$ does not divide the degree of any irreducible complex character of $G$ if and only if $P$ is abelian and normal in $G$. This result was inspired by Brauer's height zero conjecture, recently proved in \cite{Ruh25,Mal-Nav-Sch-Tie24}, which provides a representation theoretic characterisation of the commutativity of defect $p$-subgroups, and in particular of Sylow $p$-subgroups.

The normality condition of the Sylow $p$-subgroup $P$ in the It\^o--Michler theorem has also been isolated \cite{Mic86, Mal-Nav12, Gia-Law-Lon-Val22, Mor-Sch24}. Interestingly, all of these characterisations are obtained by imposing various restrictions on the divisibility of character degrees, in the spirit of the It\^o--Michler theorem.
%Interestingly, all these characterisations share a common theme: the condition is still that the prime $p$ does not divide the degree of \emph{certain} irreducible (ordinary or Brauer) characters.
A characterisation of a completely different nature has recently been conjectured by G. Malle, G. Navarro, and P. H. Tiep \cite{Mal-Nav-Tie23}. Below, we denote by $\IBr(G)$ the set of $p$-Brauer characters of $G$.

\begin{conj}[Malle--Navarro--Tiep]\label{conj:MNT}
Let $G$ be a finite group, $p$ a prime number, and $P$ a Sylow $p$-subgroup of $G$. Then $P$ is normal in $G$ if and only if $|\IBr(G)|=|\IBr(\norm G P)|$.
\end{conj}

Before proceeding further, we remark that the statement of Conjecture \ref{conj:MNT} can be rephrased in purely group theoretic terms. In fact, the number of $p$-Brauer characters of a group coincides with the number of its $p$-regular conjugacy classes. One might wonder whether a completely group theoretic proof of this conjecture can be found. We believe this to be an extremely challenging task. As a matter of fact, we show that Conjecture \ref{conj:MNT} follows from a statement (see Conjecture \ref{conj:Inductive Alperin bound} below) that is closely related to the inductive condition for the blockwise Alperin weight conjecture \cite{Nav-Tie11, Spa13I}, a condition that lies at the heart of modern finite group representation theory. Recall that a group $S$ is said to be \textit{involved} in $G$ if there is a subgroup $H$ of $G$ and a normal subgroup $N$ of $H$ such that $S$ is isomorphic to the quotient $H/N$.

%Our first main result is to give a reduction theorem for this conjecture.

\begin{theoA}
\label{thm:Main Reduction normal Sylow}
%Suppose that Conjecture \ref{conj:Inductive Alperin bound} holds at the prime $p$ for every block of maximal defect of every covering group of any non-abelian simple group of order divisible by $p$ involved in $G$. If $P$ is a Sylow $p$-subgroup of $G$, then $P\unlhd G$ if and only if $|\ibr(G)|=|\ibr(\n_G(P))|$. 
Let $G$ be a finite group, $p$ a prime number, and assume that Conjecture \ref{conj:Inductive Alperin bound} holds at the prime $p$ for every $p$-block of maximal defect of every covering group of any non-abelian finite simple group of order divisible by $p$ involved in $G$. Then the Malle--Navarro--Tiep Conjecture \ref{conj:MNT} holds for the group $G$ with respect to the prime $p$.
\end{theoA}

We are then left to verify the hypothesis of Theorem \ref{thm:Main Reduction normal Sylow}. By extending the results of \cite{Mal-Nav-Tie23}, we are able to verify Conjecture \ref{conj:Inductive Alperin bound} for $2$-blocks of maximal defect of quasi-simple groups (see Theorem \ref{thm:inductive alperin bound for p=2, max defect}), and hence to obtain the Malle--Navarro--Tiep Conjecture \ref{conj:MNT} for the prime $2$.

\begin{theoA}\label{thm:Main, p=2}
The Malle--Navarro--Tiep Conjecture \ref{conj:MNT} holds for all finite groups with respect to the prime $p=2$.
\end{theoA}

Before proceeding further, we remark that the statement of Conjecture \ref{conj:Inductive Alperin bound}, considered in the hypothesis of Theorem \ref{thm:Main Reduction normal Sylow}, is in fact a consequence of the inductive condition for the blockwise Alperin weight conjecture mentioned above (see Remark \ref{rmk:BAWC implies iBound} (ii)). The latter, has been studied by several authors (see \cite{Fen19}, \cite{Li21}, \cite{Fen-Mal22}, \cite{Fen-Li-Zha22a}, \cite{Fen-Li-Zha22b}, \cite{Fen-Li-Zha23}, \cite{An-His-Lub24}, and the survey \cite{Fen-Zha22}) and has been verified for several families of finite simple groups. As a corollary of these results, Theorem \ref{thm:Main Reduction normal Sylow} implies that Conjecture \ref{conj:MNT} also holds, at every prime $p$, for groups with abelian Sylow $2$-subgroups or abelian Sylow $3$-subgroups, and in many other cases (see, for instance, the results of \cite[Section 8]{MRR}).

%The condition on simple groups, Conjecture \ref{conj:Inductive Alperin bound} (which we state as a condition on arbitrary finite groups), is in fact \emph{implied} by the inductive blockwise Alperin weight conjecture, and in particular it is known to hold for several families of simple groups (see \cite{Fen-Zha22} and the references therein). In particular, Theorem \ref{thm:Main Reduction normal Sylow} implies that Conjecture \ref{conj:MNT} also holds at every prime for groups with abelian Sylow $2$-subgroups, among other cases (see also \cite[Section 8]{MRR}). 

In the remaining part of this introduction we explain how Theorem \ref{thm:Main Reduction normal Sylow} is obtained and describe some by-products of independent interest. The proof of Theorem \ref{thm:Main Reduction normal Sylow} can roughly be divided into two parts. In the first part, we prove a reduction theorem for Conjecture \ref{conj:Inductive Alperin bound} to finite simple groups. The statement of Conjecture \ref{conj:Inductive Alperin bound} provides an inductive condition for the blockwise version of Alperin's lower bound \cite[Consequence 1]{Alp87} (i.e., that $|\IBr(B)|\geq|\IBr(b)|$ if $b$ is the Brauer correspondent block of $B$), adapted to arbitrary finite groups. The next result is in line with a principle advertised by Dade in the 1990s and pursued in \cite{Nav-Spa14I}, \cite{Ros-iMcK}, \cite{MRR}, and \cite{Ros-CTC}: namely, Dade's principle asserts that the local-global counting conjectures should admit a refinement that can be proved to hold for arbitrary finite groups once verified for (quasi-)simple groups. In the case of the McKay conjecture, considered in \cite{Ros-iMcK}, this idea plays a crucial role in the final proof obtained by Cabanes and Sp\"ath \cite{CS}. We can state our reduction theorem as follows.

\begin{theoA}
\label{thm:Main reduction}
Let $G$ be a finite group, $p$ a prime number, and assume that Conjecture \ref{conj:Inductive Alperin bound} holds at the prime $p$ for every covering group of any non-abelian finite simple group of order divisible by $p$ involved in $G$. Then Conjecture \ref{conj:Inductive Alperin bound} holds for $G$ at the prime $p$.
\end{theoA}

We remark here that the proof of Theorem \ref{thm:Main reduction} is compatible with the choice of well-behaved classes of $p$-blocks. In particular, if we want to obtain Conjecture \ref{conj:Inductive Alperin bound} for blocks of maximal defect, then it suffices to verify it for the same class of blocks for the covering groups of finite simple groups (see Remark \ref{rem:maximal defect}). Thanks to this observation, in the second part of our proof of Theorem \ref{thm:Main Reduction normal Sylow}, we can apply Theorem \ref{thm:Main reduction} to conclude that Conjecture \ref{conj:Inductive Alperin bound} holds for blocks of maximal defect of arbitrary finite groups. This, in turns, is what allows us to conclude that groups which satisfy the equality $|\ibr(G)|=|\ibr(\n_G(P))|$ must contain a normal Sylow $p$-subgroup (see Theorem \ref{thm:proof of main, normal sylow}). It is also important to point out that, while the conclusion of Theorem \ref{thm:Main Reduction normal Sylow} does not contain any information about $p$-blocks, our proof requires a condition, Conjecture \ref{conj:Inductive Alperin bound}, that actually deals with $p$-blocks (see the proof of Theorem \ref{thm:simple groups for normal sylows}). This is also the reason why, in order to obtain Theorem \ref{thm:Main, p=2}, we need to extend the results of \cite{Mal-Nav-Tie23} and to verify Conjecture \ref{conj:Inductive Alperin bound} for $2$-blocks of maximal defect.

To conclude, observe that Theorem \ref{thm:Main reduction} provides, in particular, a reduction theorem for the blockwise version of Alperin's lower bound. This, together with the verification of Conjecture \ref{conj:Inductive Alperin bound} for $2$-blocks of maximal defect of quasi-simple groups (see Theorem \ref{thm:inductive alperin bound for p=2, max defect}), allows us to prove a blockwise version of the main theorem of \cite{Mal-Nav-Tie23}.

\begin{theoA}\label{thm:Main, blockwise lower bound p=2, max}
Let $G$ be a finite group, $B$ a $2$-block of maximal defect of $G$, and $b$ its Brauer correspondent. Then we have $|\IBr(B)|\geq |\IBr(b)|$.
\end{theoA}

The paper is structured as follows. To start, we prove Theorem \ref{thm:Main reduction} in Section \ref{sec:Reduction}. In Section \ref{sec:Reduction, blockfree}, we introduce a block-free version  of Conjecture \ref{conj:Inductive Alperin bound} (which is just the \emph{Sylow-AWC} condition from \cite{Mal-Nav-Tie23}), explain how to obtain an analogous reduction theorem and how Conjecture \ref{conj:Inductive Alperin bound} for the blocks of maximal defect implies its block-free version. Section \ref{sec:Simple groups} contains a necessary result on simple groups for Theorem \ref{thm:Main Reduction normal Sylow}. Finally, Theorem \ref{thm:Main Reduction normal Sylow} is proved in Section \ref{sec:Reduction, normal Sylow}, and Conjecture \ref{conj:Inductive Alperin bound} is checked for every $2$-block of maximal defect of every quasi-simple group in Section \ref{sec:Verification for p=2}, finishing the proof of Theorems \ref{thm:Main, p=2} and Theorem \ref{thm:Main, blockwise lower bound p=2, max}.

\section{Alperin's lower bound and proof of Theorem \ref{thm:Main reduction}}
\label{sec:Reduction}

We fix now and throughout the paper a prime number $p$. We use standard notation from character theory and the theory of $p$-blocks following \cite{Isa76} (see also \cite{Nav18}) and \cite{Nav98}. 

Let $G$ be a finite group, $B$ a block of $G$ with defect subgroup $D$, and $b$ its Brauer correspondent in $\n_G(D)$. Alperin's lower bound, introduced in \cite[Consequence 1]{Alp87} as a consequence of Alperin's weight conjecture, posits the inquality $|\ibr(b)|\leq |\ibr(B)|$. A block-free version of this bound has recently been studied in \cite{Mal-Nav-Tie23}. In particular, in \cite[Theorem 2]{Mal-Nav-Tie23}, the authors prove a reduction theorem of the block-free version of Alperin's lower bound assuming a condition on finite simple groups called \textit{Sylow-AWC} (see \cite[Section 2]{Mal-Nav-Tie23}).

We now state a blockwise version of the Sylow-AWC condition which we restate using the language of isomorphisms of character triples. Throughout this paper we use the notion of central and block isomorphisms of character triples, denoted by $\isoc$ and $\isob$ respectively, as defined in \cite[Section 3]{MRR} (see also \cite{Spa-Val}). Both $\isoc$ and $\isob$ are order relations on modular character triples that were inspired by their ordinary analogues first introduced in \cite{Nav-Spa14I}. We refer the reader to the survey \cite{Spa17I} for further details on the theory of isomorphisms of character triples.

\begin{conj}
\label{conj:Inductive Alperin bound}
Let $G\unlhd A$ be finite groups and consider a block $B$ of $G$ with defect group $D$. There is an $\n_A(D)$-invariant subgroup $\n_G(D)\leq M\leq G$, with $M<G$ whenever $D$ is not normal in $G$, such that, if $b$ denotes the Brauer correspondent of $B$ in $M$, then there exists an $\n_A(D)_B$-invariant injection
\[\Omega:\IBr(b)\hookrightarrow \IBr(B)\]
such that
\[\left(A_\chi,G,\chi\right)\isob\left(M\n_A(D)_\vartheta,M,\vartheta\right)\]
for every $\vartheta\in\ibr(b)$ and $\chi=\Omega(\theta)$
\end{conj}

Before proceeding further we make two remarks about the above condition.

\begin{rem}
\label{rmk:BAWC implies iBound}
\begin{enumerate}
\item First of all, notice that if we replace the isomorphism $\isob$ with its central version $\isoc$ and ask for an injection $\Omega:\ibr(\n_G(P))\hookrightarrow \ibr(G)$, with $P$ a Sylow $p$-subgroup of $G$, then we obtain an equivalent formulation of the Sylow-AWC condition considered in \cite{Mal-Nav-Tie23};
\item Next, as remarked in \cite[Section 2]{Mal-Nav-Tie23} for the Sylow-AWC, we observe that Conjecture \ref{conj:Inductive Alperin bound} is implied by the inductive condition for the blockwise Alperin weight conjecture. To see this, consider the formulation of the latter condition given in \cite[Conjecture 4.1]{MRR}. If $D$ is a defect group of $B$ and $\psi$ is an irreducible Brauer character of the Brauer correspondent $b$ of $B$, then we notice that $(D, \psi)$ is a $p$-weight. If $\Psi$ is the bijection given by \cite[Conjecture 4.1]{MRR} (denoted by $\Omega$ in that paper), then we deduce from the properties of block isomorphisms of modular character triples that the preimage, say $\chi$, of (the $G$-orbit of) $(D,\psi)$ under $\Psi$ belongs to $\ibr(B)$. We then obtain a bijection with the properties required in Conjecture \ref{conj:Inductive Alperin bound} by setting $\Omega(\psi):=\chi$.  
\end{enumerate}
\end{rem}

Our aim is now to prove a reduction theorem for Conjecture \ref{conj:Inductive Alperin bound} to finite quasi-simple groups.

\begin{thm}
\label{thm:Reduction}
Let $G$ be a finite group, $p$ a prime number, and suppose that Conjecture \ref{conj:Inductive Alperin bound} holds at the prime $p$ for all blocks of every covering group of any non-abelian finite simple group of order divisible by $p$ involved in $G$. Then Conjecture \ref{conj:Inductive Alperin bound} holds at the prime $p$ for all blocks of $G$.
\end{thm}

We prove Theorem \ref{thm:Reduction} using induction. For this, suppose that Theorem \ref{thm:Reduction} fails to hold for a choice of finite groups $G\unlhd A$ with respect to a block $B$ of $G$. We assume that $G$ and $A$ have been minimised with respect to $|G:\z(G)|$ first and $|A|$ then. With these choices, Conjecture \ref{conj:Inductive Alperin bound} holds for all blocks every $X\unlhd Y$ such that every non-abelian finite simple group of order divisible by $p$ involved in $X$ is also involved in $G$, and either $|X:\z(X)|<|G:\z(G)|$ or $|X:\z(X)|=|G:\z(G)|$ and $|Y|<|A|$.

\begin{rem}
\label{rmk:Removing intermediate group}
If $X\unlhd Y$ are as described above, then more is true. Namely, Conjecture \ref{conj:Inductive Alperin bound} holds for every block $C$ of $X$ with defect group $\delta$ and with respect to $M:=\n_X(\delta)$. For this, we follow the argument used in the proof of \cite[Lemma 4.1]{Ros-iMcK}. In fact, by the previous paragraph we know that Conjecture \ref{conj:Inductive Alperin bound} holds for the block $C$ of $X$ and so we can find an $\n_Y(\delta)$-invariant subgroup $\n_X(\delta)\leq M_1\leq X$, with $M_1<X$ whenever $\delta$ is not normal in $X$, and an $\n_Y(\delta)_C$-equivariant injection $\IBr(C_1)\hookrightarrow\IBr(C)$ inducing block isomorphisms of modular Brauer characters and where $C_1$ is the Brauer correspondent of $C$ in $M_1$. Now, if $\delta$ is normal in $X$, then $M_1=\n_X(\delta)$ and we are done. If $\delta$ is not normal in $X$, then we get $M_1<X$. We can now apply the same argument replacing $X\unlhd Y$ with $M_1\unlhd M_1\n_A(\delta)$ and obtain a subgroup $\n_X(\delta)=\n_{M_1}(\delta)\leq M_2\leq M_1$ and an injection $\IBr(C_2)\hookrightarrow\IBr(C_1)$ inducing block isomorphisms of modular character triples and where $C_2$ is the Brauer correspondent of $C$ in $M_2$. Since $X$ is a finite group, reiterating this process we will find a positive integer $n$ such that $M_n=\n_X(\delta)$, showing that Conjecture \ref{conj:Inductive Alperin bound} holds for $X\unlhd Y$, the block $C$, and with respect to $M=\n_X(\delta)$.
\end{rem}

In what follows we keep $G$, $A$ and $B$ as described above and divide the proof into several steps. Notice furthermore that, arguing as in the proof of \cite[Lemma 6.3 and Corollary 6.4]{MRR}, we must have $\o_p(G)=1$. Observe also that it is no loss of generality to assume that $B$ is $A$-invariant from which we deduce that $A=G\n_A(D)$ using a Frattini argument, where $D$ is a defect group of $B$. Below we will denote by $\dzo(G)$ the set of irreducible Brauer characters $\psi$ of $G$ whose corresponding character $\overline{\psi}$ of the quotient $G/\o_p(G)$ belongs to a block of defect zero.

\begin{lem}
\label{lem:Normal KD}
Let $K$ be a non-central subgroup of $G$ with $K\unlhd A$ and consider the Brauer correspondent $b'$ of $B$ in $K\n_G(D)$. For every $A$-invariant $\vartheta\in\dzo(K)$ there exists an $\n_A(D)_B$-equivariant injection
\[\Upsilon_\vartheta:\IBr(b'\mid \vartheta)\hookrightarrow \IBr(B\mid \vartheta)\]
such that
\[\left(A_\eta, G, \eta\right)\isob\left(K\n_A(D)_\eta, K\n_G(D), \Upsilon_\vartheta(\eta)\right)\]
for every $\eta\in\IBr(b'\mid \vartheta)$.
\end{lem}

\begin{proof}
Let $\P$ be a projective representation associated to the triple $(A,K,\vartheta)$ and let $\widehat A$ be the $p'$ central extension of $A$ by $Z$ from \cite[Lemma 3.12]{MRR}. Let $\eps:\hat A\rightarrow A$ be the correpsonding epimorphism, and write $\widehat H=\eps^{-1}(H)$ for any subgroup $H\leq A$. Then $\hat{K}=K_0\times Z$, where $K_0$ is isomorphic to $K$ via the restriction of $\eps_{K_0}$. Notice that $\vartheta_0:=\vartheta\circ \eps_{K_0}\in\dzo(K_0)$. Furthermore by the same result, $\vartheta_0$ extends to $\tilde\vartheta\in\IBr(\widehat A)$.  Notice that every non-abelian finite simple group involved in $\widehat{G}$ is also involved in $G$. Write $\overline{G}:=\widehat{G}/K_0$, $\overline{A}:=\widehat{A}/K_0$ and $\overline{D}:=DK_0/K_0$. Write ${H}=\norm{\widehat{G}}{\widehat{D}}=\widehat{\norm G D}$ and notice that $\overline{H}=\norm{\overline{G}}{\overline{D}}$.

 Let $\widehat{B}$ be the unique block of $\widehat{G}$ dominating $B$ (where we are identifying $G$ with $\widehat{G}/Z$). By \cite[Theorem 9.9(c)]{Nav98}, $\widehat{B}$ has defect group $D_0\in\Syl_p(\widehat{D})$, and since every character in $\widehat{B}$ lies over $1_Z$ we have $\IBr(\widehat{B}\mid\vartheta_0)=\IBr(\widehat{B}\mid\vartheta_0\times 1_Z)$ which can be naturally identified with $\IBr(B\mid\vartheta)$. Furthermore, the unique block of $\norm{\widehat{G}}{\widehat{D}}=\norm{\widehat G}{D_0}$ dominating $b'$ is its Brauer correspondent by \cite[Proposition 2.4(b)]{Nav-Spa14I} and similarly as before we obtain that $\IBr(b'\mid \vartheta_0)=\IBr(b'\mid\vartheta_0\times 1_Z)$ can be naturally identified with $\IBr(b\mid\vartheta)$.

By Gallagher's theorem and \cite[Corollary 1.5]{Mur96} there is a set of blocks $\mathcal{B}'\sbs\Bl(\overline{H})$ such that $$\IBr(\widehat b \mid \vartheta_0)=\coprod_{\overline{b}\in\mathcal{B}'}\{\psi\tilde\vartheta_{H}\mid\ \psi\in \IBr(\overline{b})\}.$$

Now if $\overline{b}\in\mathcal{B}'$ then by \cite[Theorem 1.4]{Mur96} a defect group $Q$ of $b$ is contained in $\overline{D}$. Since $\overline{D}\normal \overline{H}$ it follows from \cite[Theorem 4.8]{Nav98} that $\overline{D}\sbs Q$ and we conclude that $\overline{D}$ is a defect group of $\overline{b}$. Now using \cite[Corollary 2.5]{Mur96} we have that
if $\mathcal{B}$ denotes the set of Brauer correspondents of the blocks in $\mathcal{B}'$ in $\overline{G}$ then
$$\IBr(\widehat{B}\mid\vartheta_0)\supseteq\coprod_{\overline{B}\in\mathcal{B}}\{\varphi\tilde\vartheta_{\widehat{G}}\mid\ \varphi\in \IBr(\overline{B})\}$$
(the fact that these blocks are Brauer correspondents is not contained in the statement of  \cite[Corollary 2.5]{Mur96} but in its proof). 
 
 Notice that $|\overline{G}:\zent{\overline{G}}|\leq |G:K\zent G|<|G:\zent G|$. Therefore for any block $\overline{B}\in\mathcal{B}$ with Brauer correspondent $\overline{b}\in\mathcal{B}'$ there is an $\norm{\overline{A}}{\overline{D}}$-equivariant injection
 $$\Omega_{\overline{B}}:\IBr(\overline{b})\hookrightarrow\IBr(\overline{B})$$ with 
 \begin{equation}\label{eq:quotient isom}(\overline{A}_{\overline{\chi}}, \overline{G},\overline\chi)\isob(\norm{\overline{A}}{\overline{D}}_{\overline\psi},\norm{\overline{G}}{\overline{D}},\overline\psi)\end{equation}
 where $\overline{\chi}=\Omega_{\overline{B}}(\overline{\psi})$.

We define $$\widehat\Upsilon_\vartheta:\IBr(\widehat{b}\mid\vartheta_0)\hookrightarrow\IBr(\widehat{B}\mid\vartheta_0)$$
by $\widehat\Upsilon_\vartheta(\psi\tilde\vartheta_{H})\mapsto \Omega_{\bl({\psi})}(\psi)\tilde\vartheta_{\widehat{G}}$, where $\bl({\psi})$ denotes the block of $\psi$ in $\overline{H}$. It is straightforward to check that $\widehat\Upsilon_{\vartheta}$ is a well-defined injection. Further, since $\tilde{\vartheta}_H$ is $\norm{\widehat{A}}{\widehat{D}}$-invariant (because it is a restriction of a character in $\IBr(\widehat{A})$), then $\widehat\Upsilon_{\vartheta}$ is also $\norm{\widehat{A}}{\widehat{D}}_{\widehat{B}}$-equivariant. By \cite[Lemma 3.16]{MRR} applied to the isomorphisms from \ref{eq:quotient isom}, we have that if $\mu\in\IBr(\widehat{b}\mid\vartheta_0)$ and $\eta=\widehat\Upsilon_\vartheta(\mu)$ then
$$(\widehat{A}_\eta,\widehat{G},\eta)\isob(\norm{\widehat{A}}{\widehat{D}}_\mu,\norm{\widehat{G}}{\widehat{D}},\mu).$$
Finally, as explained in the second paragraph of this proof, $\IBr(\widehat{b}\mid\vartheta_0)$ and $\IBr(\widehat{B}\mid\vartheta_0)$ can be naturally identified with $\IBr(\widehat{b}\mid\vartheta)$ and $\IBr(\widehat{B}\mid\vartheta)$, so under this identification $ \widehat\Upsilon_\vartheta$ induces an $\norm A D_B$-equivariant injection $\Upsilon_\vartheta:\IBr(b\mid\vartheta)\hookrightarrow\IBr(B\mid\vartheta)$. Using that $\cent{\widehat{A}}{\widehat{G}}/Z=\cent A G$ by \cite[Lemma 3.12]{MRR}, we can apply \cite[Lemma 3.15]{MRR} to conclude that for all $\varphi\in\IBr(b\mid\vartheta)$ and $\chi=\Upsilon_\vartheta(\varphi)$ we have
$$(A_\chi,G,\chi)\isob(\norm A D_\varphi,\norm G D,\varphi)$$
as desired.
\end{proof}

\begin{pro}
\label{prop:Normal KD}
Let $K$ be a non-central subgroup of $G$ with $K\unlhd A$ and consider the Brauer correspondent $b'$ of $B$ in $K\n_G(D)$. Then there exists an $\n_A(D)$-equivariant injection
\[\Upsilon:\IBr(b'\mid \dzo(K))\hookrightarrow \IBr(B)\]
such that
\[\left(A_\eta, G, \eta\right)\isob\left(K\n_A(D)_\eta, K\n_G(D), \Upsilon(\eta)\right)\]
for every $\eta\in\IBr(b')$.
\end{pro}

\begin{proof}
Let $\mathcal{T}$ be an $\n_A(D)$-transversal in the set of those $\vartheta\in\dzo(K)$ such that $b'$ covers $\bl(\vartheta)$. Observe then that $D$ is a a defect group of the Fong--Reynolds correspondent $B_\vartheta$ of $B$ over $\bl(\vartheta)$ and that $G_\vartheta=G_{\bl(\vartheta)}$. Now, for any $\vartheta\in\mathcal{T}$, we have $|G_\vartheta:\z(G_\vartheta)|\leq |G:K\z(G)|<|G:\z(G)|$ and every non-abelian finite simple group involved in $G_\vartheta$ is also involved in $G$. Therefore, we can apply Lemma \ref{lem:Normal KD} with $K:=K$, $G:=G_\vartheta$ and $A:=A_\vartheta$ to obtain an $\n_A(D)_\vartheta$-equivariant injection
\[\Upsilon_\vartheta^{G_\vartheta}:\ibr(b_\vartheta'\mid\vartheta)\hookrightarrow\ibr(B_\vartheta\mid \vartheta)\]
inducing block isomorphisms of modular characters triples and where $b'_\vartheta$ is the Fong--Reynolds correspondent of $b'$ over $\bl(\vartheta)$. We then obtain an $\n_A(D)$-equivariant injection
\[\Upsilon_\vartheta^G:\ibr(b'\mid \vartheta)\hookrightarrow\ibr(B\mid \vartheta)\]
which induces block isomorphisms of modular character triples thanks to \cite[Lemma 3.8]{MRR}. Consider now an $\n_A(D)_\vartheta$-transversal $\mathcal{T}_\vartheta$ in $\ibr(b'\mid \vartheta)$ and notice then that
\[\mathbb{T}:=\coprod\limits_{\vartheta\in\mathcal{T}}\mathcal{T}_\vartheta\]
is an $\n_A(D)$-transversal in $\ibr(b'\mid \dzo(K))$. We define
\[\Upsilon\left(\psi^x\right):=\Upsilon_\vartheta^G(\psi)^x\]
for every $x\in \n_A(D)$ and any $\vartheta\in\mathcal{T}$ and $\psi\in\mathcal{T}_\vartheta$. To conclude it is enough to show that $\Upsilon$ is injective. For this it suffices to show that $\ibr(B\mid \vartheta)$ and $\ibr(B\mid \eta)$ are disjoint whenever $\vartheta,\eta\in\mathcal{T}$ with $\vartheta\neq \eta$. Suppose this is not the case. Then we can find $g\in G$ such that $\vartheta^g=\eta$. Recall, by the choice of the transversal $\mathcal{T}$, that $D$ is a defect group of the Fong--Reynolds correspondent $B_\vartheta$ of $B$ over $\bl(\vartheta)$ in $G_\vartheta$ and of the Fong--Reynolds correspondent $B_\eta$ of $B$ over $\bl(\eta)$ in $G_\eta$. Since $\vartheta^g=\eta$, we deduce that $(B_\vartheta)^g=B_\eta$ and therefore $D^g$ and $D$ are defect groups of $B_\eta$. Then, there is some $h\in G_\eta$ such that $D^{gh}=D$. We conclude that $\vartheta^{gh}=\eta$ for $gh\in\n_G(D)$, a contradiction. This proves that the map $\Upsilon$ constructed above is injective and the proof is now complete.
\end{proof}

\begin{cor}
\label{cor:Normal KD}
Let $K$ be a non-central subgroup of $G$ with $K\unlhd A$ and $K\cap D\leq \z(K)$. Then $G=K\n_G(D)$. In particular $A=G\n_A(D)=K\n_A(D)$.
\end{cor}

\begin{proof}
Let $b$ and $b'$ be the Brauer correspondents of $B$ in $\n_G(D)$ and $K\n_G(D)$ respectively. If $G>K\n_G(D)$, then $|G:\z(G)|>|K\n_G(D):\z(K\n_G(D))|$ and the inductive hypothesis yields an $\n_A(D)$-equivariant injection
\[\Phi_K:\IBr(b)\hookrightarrow\IBr(b')\]
inducing block isomorphisms of modular character triples. Consider now the injection $\Upsilon_K:\ibr(b'\mid \dzo(K))\hookrightarrow \ibr(B)$ from Proposition \ref{prop:Normal KD} and observe that $\ibr(b')=\ibr(b'\mid \dzo(K))$ since $K\cap D\leq \z(K)$. In fact, suppose that $\psi\in\ibr(b')$ lies above some $\vartheta\in\ibr(K)$. We may assume that $\vartheta$ belongs to a block of $K$ with defect group $K\cap D\leq \z(K)$. Then $K\cap D=\o_p(K)$ and we conclude that $\vartheta\in\dzo(K)$. This proves our claim. Now, combining $\Phi_K$ with $\Upsilon_K$ we conclude that Conjecture \ref{conj:Inductive Alperin bound} holds for $B$, a contradiction. Hence $G=K\n_G(D)$ as required.
\end{proof}

\begin{lem}
\label{lem:Lifting bijections}
Let $N\unlhd A$ be finite groups, consider a subgroup $A_0\leq A$ with $A=NA_0$, and set $H_0:=H\cap A_0$ for every $H\leq A$. Let $\mathcal{S}$ and $\mathcal{S}_0$ be $A_0$-stable subsets of irreducible Brauer characters of $N$ and $N_0$ respectively and assume there exists an $A_0$-equivariant bijection
\[\Psi:\mathcal{S}_0\to\mathcal{S}\]
such that
\[\left(A_\vartheta,N,\vartheta\right)\isob\left(A_{0,\vartheta_0},N_0,\vartheta_0\right)\]
for every $\vartheta_0\in\mathcal{S}_0$ and $\vartheta:=\Psi(\vartheta_0)$. Then, for every $N\leq J\leq A$, there exists an $\n_{A_0}(J)$-equivariant bijection
\[\Phi_J:\IBr\left(J_0\enspace
\middle|\enspace\mathcal{S}_0\right)\to \IBr\left(J\enspace
\middle|\enspace\mathcal{S}\right)\]
such that
\[\left(\n_A(J)_\chi,J,\chi\right)\isob\left(\n_{A_0}(J)_{\chi_0},J_0,\chi_0\right)\]
for every $\chi_0\in\IBr(J_0\mid \mathcal{S}_0)$ and $\chi:=\Phi_J(\chi_0)$.
\end{lem}

\begin{proof}
This follows by arguing as in the proof of \cite[Proposition 2.10]{Ros22} but applying the results on block isomorphisms of modular character triples from \cite[Section 3]{MRR} (see also the proof of \cite[Theorem 4.5]{MRR}).
\end{proof}

\begin{pro}
\label{pro:Non-central defect}
Let $K$ be a non-central subgroup of $G$ with $K\unlhd A$. Then, $K\cap D$ is not contained in $\z(K)$.
\end{pro}

\begin{proof}
Consider a block $B'$ of $K$ covered by $B$. Without loss of generality we may assume that $K\cap D$ is a defect group of $B'$. Assume now that $K\cap D\leq \z(K)$. Then $KD$ is a normal subgroup of $A$ thanks to Corollary \ref{cor:Normal KD}. Furthermore, in this case, \cite[(1.ex.1)]{Bro-Pug80} shows that the block $B'$ is nilpotent. Since $KD/K$ is a $p$-group, \cite[Corollary 9.6]{Nav98} implies that there exists a unique block $C$ of $KD$ covering $B'$, and which is in turn covered by $B$. Furthermore, by applying \cite[Theorem 2]{Cab87} (see also the remark on non-stable blocks at the end of that proof) we conclude that $C$ is also nilpotent. We denote by $c$ the Brauer correspondent of $C$ in $\n_{KD}(D)$, also a nilpotent block, and let $\psi$ and $\varphi$ be the unique Brauer characters belonging to $C$ and $c$ respectively. Now, according to \cite[Lemma 8.9]{MRR} there exists a block isomorphism of modular character triples
\begin{equation}
\label{eq:Non-central defect}
\left(A_\psi,KD,\psi\right)\isob\left(\n_A(D)_\varphi,\n_{KD}(D),\varphi\right).
\end{equation}
We claim that $\psi$ is $A$-invariant, and therefore $\varphi$ is $\n_A(D)$-invariant. For this, notice that $\IBr(B)$ and $\IBr(b)$ coincide with the sets $\IBr(B\mid \psi)$ and $\IBr(b\mid \varphi)$ respectively. By applying \cite[Theorem 8.9 and Theorem 9.14]{Nav98}, we deduce that, if $B'$ and $b'$ are the Fong--Reynolds correspondents of $B$ over $\bl(\psi)$ and of $b$ over $\bl(\varphi)$ respectively, then induction of Brauer characters yields bijections between $\IBr(B'\mid \psi)$ and $\IBr(B\mid \psi)$ and between $\IBr(b'\mid \varphi)$ and $\IBr(b\mid \varphi)$. If $A_\psi<A$, then by the inductive hypothesis we obtain an injection of $\IBr(b')$ into $\IBr(B')$ with the properties described in Conjecture \ref{conj:Inductive Alperin bound}. Since $\IBr(B')$ and $\IBr(b')$ coincide with $\IBr(B'\mid \psi)$ and $\IBr(b'\mid \varphi)$ respectively, we can apply \cite[Lemma 3.8]{MRR} to obtain an injection of $\IBr(b)$ into $\IBr(B)$ inducing block isomorphisms of modular character triples as required by Conjecture \ref{conj:Inductive Alperin bound}. This contradicts our choice of $B$ and therefore $\psi$ must be $A$-invariant. Since $A_\psi=A_C$ and $\varphi$ is the unique Brauer character belonging to the Brauer correspondent $c$ of $C$ we deduce that $\n_A(D)=\n_A(D)_C=\n_A(D)_c=\n_A(D)_\varphi$. This proves our claim.

Now, recalling \eqref{eq:Non-central defect}, we can apply Lemma \ref{lem:Lifting bijections}, with $A:=A$, $A_0:=\n_A(D)$, $N:=KD$, $J:=G$, and where $\mathcal{S}=\{\psi\}$ and $\mathcal{S}_0=\{\varphi\}$, to obtain an $\n_A(D)$-equivariant bijection
\[\Omega:\IBr(\n_G(D)\mid \varphi)\to \IBr(G\mid \psi)\] 
inducing block isomorphisms of modular character triples. In particular, $\Omega$ restricts to a bijection between $\IBr(b)$ and $\IBr(B)$ and so Conjecture \ref{conj:Inductive Alperin bound} holds for the block $B$ against our previous assumption. We conclude that $K\cap D$ is not contained in $\z(K)$.
\end{proof}

\begin{lem}
\label{lem:Isomorphic simple groups, lifting}
Let $K\unlhd A$ be finite groups with $K$ perfect and assume that $p$ does not divide the order of $\z(K)$ and that $K/\z(K)$ is the direct product of copies of a non-abelian simple group $S$ of order divisible by $p$. Let $X$ be the universal $p'$-covering group of $S$. If Conjecture \ref{conj:Inductive Alperin bound} holds with respect to $X\unlhd X\rtimes \aut{X}$, then Conjecture \ref{conj:Inductive Alperin bound} holds for $K\unlhd A$.
\end{lem}

\begin{proof}
The result follows by arguing as in the proof of \cite[Corollary 4.3]{MRR}.
\end{proof}

We can finally prove our reduction theorem.

\begin{proof}[Proof of Theorem \ref{thm:Reduction}]
We consider $G\unlhd A$, $B$, and $D$ as described at the beginning of this section. With these choices recall that $\o_p(G)=1$ so that $\z(G)\leq \o_{p'}(G)$. Since $\o_{p'}(G)\cap D=1$, applying Proposition \ref{pro:Non-central defect} we deduce that $\o_{p'}(G)\leq \z(G)$. Therefore, we must have $\z(G)<\F^*(G)$, as otherwise $G$ would be abelian, and there exists a characteristic subgroup $K$ of $\F^*(G)$ such that $K$ is perfect and $K/\z(K)$ is the direct product of isomorphic copies of a non-abelian finite simple group $S$ of order divisible by $p$. Let $C$ be a block of $K$ covered by $B$ with defect group $K\cap D$. By hypothesis and using Lemma \ref{lem:Isomorphic simple groups, lifting} we get an $\n_A(D)$-invariant subgroup $\n_K(K\cap D)\leq M\leq K$ and an $\n_A(D)_C$-equivariant injection
\[\Omega_C:\IBr(c)\hookrightarrow\IBr(C)\]
inducing block isomorphisms of character triples and where $c$ is the Brauer correspondent of $C$ in $M$. Observe, using \cite[Lemma 3.8]{MRR} and \cite[Theorem 9.14]{Nav98}, that we may assume that $C$ is $G$-invariant. Now, applying Lemma \ref{lem:Lifting bijections} with $A:=A_C$, $A_0:=M\n_A(D)_C$, $J:=G$, and with $\mathcal{S}:=\Omega_C(\IBr(c))\subseteq \IBr(C)$, $\mathcal{S}_0:=\IBr(c)$, and $\Psi:=\Omega_C$, we obtain an $\n_A(D)_C$-equivariant bijection
\[\wt{\Omega}_C:\IBr(M\n_G(D)\mid \mathcal{S}_0)\to \IBr(G\mid \mathcal{S})\]
inducing block isomorphisms of modular character triples. Let $B'$ be the Brauer correspondent of $B$ in $M\n_G(D)$ and observe that $\wt{\Omega}_C$ maps Brauer characters belonging to $B'$ to Brauer characters belonging to $B$. Moreover, since each Brauer character of $B'$ lies above some element in $\mathcal{S}_0=\IBr(c)$, we deduce that $\wt{\Omega}_C$ restricts to an $\n_A(D)_C$-equivariant injection
\[\Omega_B:\IBr(B')\hookrightarrow\IBr(B)\]
inducing block isomorphisms of character triples. To conclude notice that $A=A_B=GA_C$ by a Frattini argument and therefore that $M\n_G(D)$ is $\n_A(D)$-invariant and that $\Omega_B$ is $\n_A(D)$-equivariant. This shows that Conjecture \ref{conj:Inductive Alperin bound} holds for $B$ and the proof is complete.
\end{proof}

We conclude this section by showing that the proof of Theorem \ref{thm:Reduction} can be adapted to obtain results about specific classes of blocks.

\begin{rem}\label{rem:maximal defect}
Suppose one wants to prove Conjecture \ref{conj:Inductive Alperin bound} only for blocks of maximal defect of $G$ (a similar argument can be made, for instance, for the class of principal blocks). We claim that this would follow from the argument used to prove Theorem \ref{thm:Reduction}, but assuming that only the blocks of maximal defect of all covering groups of finite simple groups of order divisible by $p$ involved in $G$ satisfy the conclusion of Conjecture \ref{conj:Inductive Alperin bound}. We inspect the proofs of all the intermediate steps. First, consider Lemma \ref{lem:Normal KD}. There we have, by assumption, that $B$ is a block of maximal defect of $G$. In particular, the $p$-subgroup $D$ considered there is a Sylow $p$-subgroup of $G$. Then, we notice that the group $\overline{D}$ constructed there is a Sylow $p$-subgroup of $\overline{G}$. Since $\overline{B}$ has defect group $\overline{D}$, we deduce that $\overline{B}$ is a block of maximal defect. Now we can apply our weaker hypothesis, which only applies to blocks of maximal defect, to obtain the bijection $\Omega_{\overline{B}}$ for the block $\overline{B}$. We can then apply this result to construct a bijection $\Upsilon$ as in Proposition \ref{prop:Normal KD} for our block of maximal defect $B$. Here the only thing to remember is that Fong--Reynolds correspondents have common defect groups (see \cite[Theorem 9.14]{Nav98}). Corollary \ref{cor:Normal KD} simply follows using induction and noticing that Brauer correspondent blocks have common defect groups. We can now proceed to the proof of Proposition \ref{pro:Non-central defect}. Here, we consider a normal subgroup $K$ of $G$ and a block $B'$ covered by $B$. Since $D$ is a Sylow $p$-subgroup of $G$ and $K$ is normal in $G$, it follows that $D\cap K$ is a Sylow $p$-subgroup of $K$. However $D\cap K$ is a defect group of $B'$ and therefore $B'$ is a block of maximal defect. We can now apply the inductive hypothesis to $B'$ and assume that the character $\psi$ considered there is $A$-invariant. The proof then follows. Next, we consider the proof of Lemma \ref{lem:Isomorphic simple groups, lifting}. We are assuming that Conjecture \ref{conj:Inductive Alperin bound} holds for the blocks of maximal defect of $X$ and want to show that it holds for the blocks of maximal defect of $K$. Since $p$ does not divide the order of $\z(K)$ nor the order of $\z(X)$, it follows that all quotients considered in the proof of \cite[Corollary 4.3]{MRR} are formed with respect to a normal $p'$-subgroup. With this in mind, an inspection of the argument used in \cite[Corollary 4.3]{MRR} shows that if we want to obtain Conjecture \ref{conj:Inductive Alperin bound} only for the blocks of maximal defect of $K$ it suffices to assume it for the blocks of maximal defect of $X$, as desired. We can now repeat the rest of the proof of Theorem \ref{thm:Reduction}.
\end{rem}

\subsection{The block-free case}\label{sec:Reduction, blockfree}

In \cite{Mal-Nav-Tie23} the authors introduced the Sylow-AWC condition, a block-free version of Conjecture \ref{conj:Inductive Alperin bound} (see Remark \ref{rmk:BAWC implies iBound} (i)). We reformulate it below in terms of central isomorphisms of modular character triples.

\begin{conj}
\label{conj:Inductive Alperin bound, blockfree}
Let $G\unlhd A$ be finite groups and consider a Sylow $p$-subgroup $P$ of $G$. There is an $\n_A(P)$-invariant subgroup $\n_G(P)\leq M\leq G$, with $M<G$ whenever $P$ is not normal in $G$, and an $\n_A(P)$-invariant injection
\[\Omega:\IBr(\n_G(P))\hookrightarrow \IBr(G)\]
such that
\[\left(A_\chi,G,\chi\right)\isoc\left(M\n_A(P)_\vartheta,M,\vartheta\right)\]
for every $\vartheta\in\ibr(\n_G(P))$ and $\chi=\Omega(\psi)$
\end{conj}

A simplified version of the arguments used in Section \ref{sec:Reduction}, can be used to obtain the following block-free version of Theorem \ref{thm:Reduction}. We include this result here for future reference.

\begin{thm}
\label{thm:Reduction, blockfree}
Let $G$ be a finite group, $p$ a prime number, and suppose that Conjecture \ref{conj:Inductive Alperin bound, blockfree} holds at the prime $p$ for every covering group of any non-abelian finite simple group of order divisible by $p$ involved in $G$. Then Conjecture \ref{conj:Inductive Alperin bound, blockfree} holds at the prime $p$ for $G$.
\end{thm}

Next, we observe that Conjecture \ref{conj:Inductive Alperin bound, blockfree} follows from Conjecture \ref{conj:Inductive Alperin bound} in the special case of blocks of maximal defect. For completeness, we give a proof of this simple fact in the next lemma.

\begin{lem}\label{lem:iAb for maximal implies}
Assume that Conjecture \ref{conj:Inductive Alperin bound} is satisfied for every block of maximal defect of $G\normal A$ and let $P$ be a Sylow $p$-subgroup of $G$. Then there is an $\norm A P$-equivariant injection
\[\Omega:\IBr(\norm G P)\hookrightarrow\IBr(G)\]
such that 
\[\left(A_{\Omega(\varphi)},G,\Omega(\varphi)\right)\isob\left(\norm A P_\varphi,\norm G P,\varphi\right)\]
for every $\varphi\in\IBr(\norm G P)$. 
\end{lem}

\begin{proof}
Let $\mathcal{T}$ be an $\norm A P$-transversal in the set of blocks of $\norm G P$. Notice that the set of blocks $b^G$, for $b\in \mathcal{T}$, is an $\norm  A P$-transversal on the set of blocks of maximal defect of $G$ by Brauer's first main theorem. For each $b\in\mathcal{T}$, Conjecture \ref{conj:Inductive Alperin bound} gives an $\norm A P_B$-equivariant injection 
\[\Omega_B:\IBr(b)\hookrightarrow \IBr(B)\]
inducing block isomorphisms of modular character triples, where $B\in\Bl(G)$ is the Brauer correspondent block of $b$. For any $\psi\in\IBr(\norm G P)$ there is some $x\in \norm A P$ with $\bl(\psi)^x\in\mathcal{T}$ and we set $\Omega(\psi)=\Omega_{\bl(\psi)^x}(\psi^x)$. This yields an $\norm A P$-equivariant injection
\[\Omega:\IBr(\norm G P)\hookrightarrow\IBr(G)\]
Furthermore, since block isomorphisms of modular character triples are, in particular, central isomorphisms, we conclude that $\Omega$ satisfies the required properties.
\end{proof}

\section{Simple groups for Theorem \ref{thm:Main Reduction normal Sylow}}\label{sec:Simple groups}

The goal of this section is to show that simple groups satisfying Conjecture \ref{conj:Inductive Alperin bound} are not counterexamples to Conjecture \ref{conj:MNT}. We begin with some considerations for alternating groups.

Recall that the $p$-blocks of $\SSS_n$ are determined by a $p$-core $\gamma$, and the characters $\chi^\lambda,\chi^\mu$ associated to the partitions $\lambda,\mu$ of $n$ are in the block determined by $\gamma$ if their $p$-core is $\gamma$ (this is known as Nakayama's conjecture, see \cite[6.1.35]{JK}). We denote this block by $B(\gamma,w)$ where $w=(n-|\gamma|)/p$ is the weight of $B(\gamma,w)$.  The defect groups of $B(\gamma,w)$ are isomorphic to Sylow $p$-subgroups of $\SSS_{pw}$ (see \cite[6.2.45]{JK}). Let $B$ be a block of $\AAA_n$ covered by $B(\gamma,w)$.  If $p\neq 2$,  by \cite[Theorem 9.26]{Nav98} we have that the defect groups of $B$ are defect groups of $B(\gamma,w)$. In the case $p=2$, using again \cite[Theorem 9.26]{Nav98}, we can find a defect group $D$ of $B$ and a defect group $Q$ of $B(\gamma,w)$ such that $Q\cap\AAA_n=D$ and $|Q:D|=2$ (notice that $Q$ is not contained in $\AAA_n$ because it contains $2$-cycles). In particular, for all primes $p$, a $p$-block of non-maximal defect of $\SSS_n$ covers $p$-blocks of non-maximal defect of $\AAA_n$. We recall the following fact.

\begin{pro}\label{pro:alternating}
The group $\AAA_n$, $n\geq 5$, has a $p$-block of non-maximal defect for any prime $p$.
\end{pro}
\begin{proof}%{\color{red}{\textbf{[CHECK GUNTER'S COMMENTS]}}}
By \cite{Gra-Ono96} we may assume $p<5$. If $p=3$ then consider the $3$-core
$$\gamma=\begin{cases}(4,2) &\text{ if } n\equiv 0 \mod 3\\
(3,1) &\text{ if } n\equiv 1 \mod 3\\
(3,1,1) &\text{ if } n\equiv 2 \mod 3
\end{cases}$$
and let $B(\gamma, w)$ be the $3$-block of $\SSS_n$ determined by $\gamma$ with weight $w$. Notice that in all cases $w < (n-k)/3$ where $k\in\{1,2,3\}$ is such that $n\equiv k \mod 3$. As explained before the statement of this result, this implies that $B(\gamma,w)$ has non-maximal defect, and so does any block $B\in\Bl(\AAA_n)$ covered by $B(\gamma,w)$.
If $p=2$ then we consider the $2$-core
$$\gamma=\begin{cases}
(2,1) &\text{ if } n \text{ odd }\\
(3,2,1) &\text{ if } n \text{ even }
\end{cases}$$
and again let $B(\gamma,w)$ be the $2$-block of $\SSS_n$ determined by $\gamma$ with weight $w$. Notice that in all cases $w<(n-k)/2$ where $n\equiv k \mod 2$ and $k\in\{0,1\}$. We conclude that $B(\gamma,w)$ has non-maximal defect, and it follows from the final comment in the paragraph before this result that any block of $\AAA_n$ covered by $B(\gamma,w)$ has non-maximal defect.
\end{proof}

We can now prove the desired result on simple groups.

\begin{thm}\label{thm:simple groups for normal sylows}
Assume $S$ is a nonabelian finite simple group of order divisible by $p$ and let $P$ be a Sylow $p$-subgroup of $S$. If Conjecture \ref{conj:Inductive Alperin bound} holds for $G=S$, then $|\IBr(S)|>|\IBr(\norm S P)|$.
\end{thm}

\begin{proof}
Assume that $|\IBr(S)|=|\IBr(\norm S P)|$. Then, Conjecture \ref{conj:Inductive Alperin bound} and Brauer's first main theorem imply that every block of $S$ has maximal defect. By Proposition \ref{pro:alternating}, $S$ is not a simple alternating group. By \cite{Mic86,Gra-Ono96} we have that $p\leq 5$, and then the main result of \cite{Wil88} implies that there are only finitely many possibilities for $S$ which are readily checked in \cite{GAP}.
\end{proof}

\begin{rem}
The inequality $|\IBr(S)|>|\IBr(\norm S P)|$ can be proven with the exact same argument assuming that the blockwise Alperin weight conjecture holds for $S$.
\end{rem}

\section{Normal Sylow subgroups}
\label{sec:Reduction, normal Sylow}

The aim of this section is to prove Theorem \ref{thm:Main Reduction normal Sylow}, which we restate for the reader's convenience.

\begin{thm}
Assume that Conjecture \ref{conj:Inductive Alperin bound} holds for every maximal defect block of every covering group of every non-abelian finite simple group of order divisible by $p$ involved in a finite group $G$ and let $P$ be a Sylow $p$-subgroup of $G$. Then $|\IBr(G)|=|\IBr(\norm G P)|$ if and only if $P\normal G$.
\end{thm}

We divide the proof of Theorem \ref{thm:Main Reduction normal Sylow} in a series of results. Notice that in every step, except in Theorem \ref{thm:proof of main, normal sylow}, we assume Conjecture \ref{conj:Inductive Alperin bound, blockfree}. Thanks to Lemma \ref{lem:iAb for maximal implies}, the latter assumption is ensured as long as Conjecture \ref{conj:Inductive Alperin bound} holds for all blocks of maximal defect. Furthermore, observe that the assumption of Conjecture \ref{conj:Inductive Alperin bound} in Theorem \ref{thm:proof of main, normal sylow} is necessary to apply Theorem \ref{thm:simple groups for normal sylows} in the last step of the proof.

\begin{pro}\label{pro:injection}
Assume that Conjecture \ref{conj:Inductive Alperin bound, blockfree} is satisfied for $G:=N$ and $A:=G$ and let $Q$ be a Sylow $p$-subgroup of $N$. Then there is an $\norm G Q$-equivariant injection
\[\Omega:\IBr(\norm N Q)\hookrightarrow \IBr(N)\]
such that for all $\varphi\in\IBr(\norm N Q)$ and all subgroups  $N\leq M\leq G$ we have
\[|\IBr(\norm M Q\mid\varphi)|=|\IBr(M\mid\Omega(\varphi)|.\]
\end{pro}

\begin{proof}
By hypothesis we have an $\norm G Q$-equivariant injection
$$\Omega:\IBr(\norm N Q)\hookrightarrow\IBr(N).$$
It follows from the definition of the order relation $\isoc$ that the modular character triples
$(\norm M Q_\varphi,\norm N Q,\varphi)$ and $(M_{\Omega(\varphi)},N,\Omega(\varphi))$ are isomorphic. This implies that $$|\IBr(\norm M Q_\varphi\mid\varphi)|=|\IBr(M_{\Omega(\varphi)}\mid\Omega(\varphi))|$$ and the result follows by applying the Clifford correspondence.
\end{proof}

The next result will allow us to \emph{go down} to a normal subgroup $N\normal G$.

\begin{pro}\label{pro:going to normal}
Assume that Conjecture \ref{conj:Inductive Alperin bound, blockfree} is satisfied for $G:=N$ and $A:=G$ and let $Q$ be a Sylow $p$-subgroup of $N$. Assume furthermore that $|\IBr(G)|=|\IBr(\norm G Q)|$. Then $|\IBr(N)|=|\IBr(\norm N Q)|$.
\end{pro}
\begin{proof}
By the Frattini argument we have $N\norm G Q=G$. Let
\[\Omega:\IBr(\norm N Q)\hookrightarrow\IBr(N)\]
be the injection from Proposition \ref{pro:injection}. Let $\Delta\sbs\IBr(\norm N Q)$ be an $\norm G Q$-transversal, and notice that $\Omega(\Delta)$ is a transversal in the image $\Xi:=\Omega(\IBr(\norm N Q))\sbs\IBr(N)$. We work to show that $\Xi=\IBr(N)$. We have that
$$|\IBr(\norm G Q)|=\sum_{\psi\in\Delta}|\IBr(\norm G Q\mid\psi)|$$ and since $|\IBr(\norm G Q\mid\psi)|=|\IBr(G\mid\Omega(\psi))|$ we have
$$|\IBr(\norm G Q)|=\sum_{\psi\in\Delta}|\IBr(G\mid\Omega(\psi))|\leq|\IBr(G)|=|\IBr(\norm G Q)|$$
so we conclude that $\sum_{\psi\in\Delta}|\IBr(G\mid\Omega(\psi))|=|\IBr(G)|$. In other words, every element of $\IBr(G)$ lies over an element in $\Omega(\IBr(\norm N Q))$. Since $\IBr(\norm N Q)$ is closed under $\norm G Q$-conjugation, its image $\Xi$ is $\norm G Q$-stable (and therefore $G$-stable by Frattini). Let $\varphi\in\IBr(N)$ and let $\chi\in\IBr(G\mid\varphi)$. We have proved that $\chi$ lies over an element of $\Xi$, and by Clifford's theorem this element must be $G$-conjugate to $\varphi$. It follows that $\varphi\in\Xi$ and we are done.
\end{proof}

\begin{lem}\label{lem:norm conjugate}
Let $N$ be a normal subgroup of $G$, $P$ a Sylow $p$-subgroup of $G$, and $\varphi,\psi\in\Irr(N)$ be $G$-conjugate characters of $N$ that are both $P$-invariant. Then $\varphi$ and $\psi$ are $\norm G P$-conjugate.
\end{lem}
\begin{proof}
Let $x\in G$ be such that $\psi=\varphi^x$. Since $G_{\psi}=G_{\varphi}^x$ we have that $G_{\psi}$ contains $P$ and $P^x$. This implies that there is some $g\in G_\psi$ with $P^{xg}=P$. Now $\varphi^{xg}=\psi^g=\psi$ and $xg\in\norm G P$, as desired.
\end{proof}

The idea of the next proof was given to us by G. Navarro. In its proof, we use \cite[Theorem 2]{Mal-Nav-Tie23}, which assumes the so-called \emph{inductive Sylow-AWC condition}. Notice that the latter is equivalent to our Conjecture \ref{conj:Inductive Alperin bound, blockfree}.

\begin{pro}[Navarro]\label{pro:glauberman}
Let $N\normal G$ with $N$ a $p'$-group and let $P$ be a Sylow $p$-subgroup of $G$. Assume that Conjecture \ref{conj:Inductive Alperin bound, blockfree}  holds at the prime $p$ for every covering group of any non-abelian finite simple group of order divisible by $p$ involved in $G$, and that $|\IBr(G)|=|\IBr(\norm G P)|$. Then for all $\lambda\in\Irr_P(N)$ we have
$$|\IBr(G\mid\lambda)|=|\IBr(N\norm G P\mid\lambda)|.$$
\end{pro}
\begin{proof}
By assumption, Conjecture \ref{conj:Inductive Alperin bound, blockfree} (i.e., the Sylow-AWC condition from \cite{Mal-Nav-Tie23}) holds at the prime $p$ for every covering group of any non-abelian finite simple group of order divisible by $p$ involved in $G$. By \cite[Theorem 2]{Mal-Nav-Tie23} we have $|\IBr(G\mid\lambda)|\geq|\IBr(N\norm G P\mid\lambda)|$. Let $\Delta_P$ be a $G$-transversal on the set $\Irr_P(N)$. %and complete $\Delta_P$ to a full $G$-transversal on $\Irr(N)$, $\Delta=\Delta_P\cup\Delta_0$ where $\Delta_0$ contains no character that is invariant under any Sylow $p$-subgroup of $G$. 

We have that
\begin{align}\label{eq:inequalities}
|\IBr(G)|&\geq\sum_{\varphi\in\Delta_P}|\IBr(G\mid\varphi)|\geq \sum_{\varphi\in\Delta_P}|\IBr(N\norm G P\mid\varphi)|.
\end{align}
%\begin{align*}
%|\IBr(G)|&=\sum_{\varphi\in\Delta}|\IBr(G\mid\varphi)|=\sum_{\varphi\in\Delta_P}|\IBr(G\mid\varphi)|+\sum_{\psi\in\Delta_0}|\IBr(G\mid\psi)|\geq\\&\geq \sum_{\varphi\in\Delta_P}|\IBr(N\norm G P\mid\varphi)|+\sum_{\psi\in\Delta_0}|\IBr(G\mid\psi)|.
%\end{align*}
Now, if $\varphi'$ denotes the Glauberman correspondent in $\cent N P=N\cap\norm G P$ of $\varphi\in\Delta_P$, then we have
$$|\IBr(N\norm G P\mid\varphi)|=|\IBr(\norm G P\mid\varphi')|$$ by the Clifford correspondence and \cite[Proposition 1.1]{Mar-Ros23}, using that $N\norm G P$ is $p$-solvable. Furthermore, by Lemma \ref{lem:norm conjugate}, $\Delta_P$ is also an $\norm G P$-transversal on $\Irr_P(N)$. By the properties of the Glauberman correspondence, this implies that $\{\varphi'\mid\varphi\in\Delta_P\}$ is an $\norm G P$-transversal in $\Irr(\cent N P)$. We conclude that
$$|\IBr(G)|\geq \sum_{\varphi\in\Delta_P}|\IBr(\norm G P\mid\varphi')|=|\IBr(\norm G P)|=|\IBr(G)|$$
which forces every inequality in \ref{eq:inequalities} to be an equality, so we are done.
\end{proof}

The following is the nontrivial direction of Theorem \ref{thm:Main Reduction normal Sylow}.

\begin{thm}\label{thm:proof of main, normal sylow}
Assume that Conjecture \ref{conj:Inductive Alperin bound} holds for every block of maximal defect of every covering group of every non-abelian finite simple group of order divisible by $p$ involved in $G$ and let $P$ be a Sylow $p$-subgroup of $G$. If $|\IBr(G)|=|\IBr(\norm G P)|$ then $P\normal G$.
\end{thm}
\begin{proof}
By Theorem \ref{thm:Reduction} and Remark \ref{rem:maximal defect} we have that Conjecture \ref{conj:Inductive Alperin bound} holds for every block of maximal defect of every subgroup of $G$. Therefore, Lemma \ref{lem:iAb for maximal implies} implies that Conjecture \ref{conj:Inductive Alperin bound, blockfree} hold for every subgroup of $G$. 

We argue by induction on $|G|$. By \cite[Theorem 23.12]{Man-Wol93} we may assume $G$ is not $p$-solvable and that $\oh{p}G=1$. Let $N\normal G$ and let $Q=P\cap N$ be a Sylow $p$-subgroup of $N$. We work to prove $Q\normal G$ so we assume otherwise. We claim that in this case $|\IBr(G)|=|\IBr(\norm G Q)|$.

We have $\norm G P\leq \norm G Q<G$ and $G=N\norm G Q$ by Frattini. By Lemma \ref{lem:iAb for maximal implies}, our hypothesis implies that \ref{conj:Inductive Alperin bound, blockfree}  holds at the prime $p$ for every covering group of any non-abelian finite simple group of order divisible by $p$ involved in $G$. Since Conjecture \ref{conj:Inductive Alperin bound, blockfree} is exactly the Sylow-AWC condition from \cite{Mal-Nav-Tie23}, we can apply \cite[Theorem 2]{Mal-Nav-Tie23} and conclude that $|\IBr(\norm G Q)|\geq|\IBr(\norm G P)|=|\IBr(G)|$. By Proposition \ref{pro:injection} there is an $\norm G Q$-equivariant injection
$$\Omega:\IBr(\norm N Q)\hookrightarrow \IBr(N)$$
with $$|\IBr(\norm G Q\mid\varphi)|=|\IBr(G\mid\Omega(\varphi)|$$ for all $\varphi\in\IBr(\norm N Q)$. Let $\Delta\sbs\IBr(\norm N Q)$ be an $\norm G Q$-transversal. We have that 
$$|\IBr(\norm G Q)|=\sum_{\psi\in\Delta}|\IBr(\norm G Q\mid\Omega(\psi))|$$
and since $\Omega(\Delta)$ is also a transversal in $\Omega(\IBr(\norm N Q))$ we have
$$|\IBr(\norm G Q)|=\sum_{\psi\in\Delta}|\IBr(G\mid\Omega(\psi)|\leq|\IBr(G)|=|\IBr(\norm G P)|\leq|\IBr(\norm G Q)|$$
which forces $|\IBr(G)|=|\IBr(\norm G Q)|$. By Proposition \ref{pro:going to normal} we have $|\IBr(\norm N Q)|=|\IBr(N)|$ so by the inductive hypothesis $Q\normal N$. Since $\oh{p}G=1$ this implies $|N|$ is not divisible by $p$ for any proper normal subgroup $N$ of $G$.

If $N=\oh{p'}G$ then we conclude that $G/N$ is simple non-abelian of order divisible by $p$, and by Proposition \ref{pro:glauberman} for $\lambda=1_N$ we have that
$$|\IBr(G/N)|=|\IBr(N\norm{G}{P}/N)|$$
and since $\norm{G/N}{PN/N}=N\norm G P/N$ this contradicts Theorem \ref{thm:simple groups for normal sylows}.
\end{proof}

\section{Verifying the condition for $2$-blocks of maximal defect}
\label{sec:Verification for p=2}

In this section, we finally verify Conjecture \ref{conj:Inductive Alperin bound} for $2$-blocks of maximal defect of quasi-simple groups. First, we give the following criterion to establish Conjecture \ref{conj:Inductive Alperin bound} for quasi-simple groups.

\begin{pro}\label{prop:iBAWC-}
	Suppose	that $S$ is a non-abelian simple group of order divisible by $p$ and $G$ is a covering group of $S$.
	Let $B$ be an $p$-block of $G$.
			Then the block $B$ satisfies Conjecture \ref{conj:Inductive Alperin bound} if the following two conditions hold.
	\begin{enumerate}[\rm(i)]
		\item Let $D$ be a defect group of $B$, and let $b\in\Bl(\N_G(D))$ be the Brauer correspondent of $B$. There exists an $\Aut(G)_B$-equivariant injection \[\Omega:\ibr(b)\hookrightarrow\ibr(B).\]
		\item For each $\psi\in\ibr(b)$ and $\chi=\Omega(\psi)$, there exist a finite group $A$ and Brauer characters $\tilde\chi\in\ibr(A)$ and $\tilde\psi\in\ibr(\N_A(D))$ such that the following statements are satisfied.
		\begin{enumerate}[\rm(a)]
			\item The group $A$ satisfies $G\unlhd A$, $A/\c_A(G)\cong \Aut(G)_\chi$, $\c_A(G)=\Z(A)$.
			\item $\tilde\chi$ is an extension of $\chi$.			 
			\item $\tilde\psi$ is an extension of $\psi$.
			\item For every intermediate subgroup $J$ with $G\le J\le A$ the Brauer characters $\tilde\chi$ and $\tilde\psi$ satisfy
			\[\bl(\Res^{\N_A(D)}_{\N_{J}(D)}(\tilde\psi))^J=\bl(\Res^{A}_{J}(\tilde\chi)).\]
		\end{enumerate}	
		\end{enumerate}
\end{pro}
	
\begin{proof}
This follows by the proof of \cite[Thm.~4.4]{Spa17I}.
\end{proof}

In Proposition~\ref{prop:iBAWC-}, if moreover $\Aut(G)_B/G$ is cyclic, then the condition (ii)  holds automatically
by the proof of \cite[Lemma~6.1]{Spa13I}.
	
\begin{lem}\label{lem:case1}
		Suppose	that $S$ is a non-abelian simple group such that the $p'$-part of its Schur multiplier $\textup{Mult}(S)_{p'}$ is trivial, and the Sylow $p$-subgroups of $S$ are self-normalizing in $S$.
		Then Conjecture \ref{conj:Inductive Alperin bound} holds for every $p$-block of $S$ of maximal defect.
\end{lem}

\begin{proof}
Since the Sylow $p$-subgroups of $S$ are self-normalizing, Brauer's third main theorem implies that the only block of maximal defect of $S$ is the principal block. Then, the map $1_{\N_S(Q)}\mapsto 1_{S}$ clearly gives the injection $\Omega$ satisfying the required conditions.
\end{proof}		
	
\begin{lem}\label{lem:case2}
	Suppose	that $S$ is a non-abelian simple group and $G$ is the universal $p'$-covering group of $S$.
Assume that the Sylow $p$-subgroups of $S$ are self-normalizing in $S$ and $\Out(G)$ acts regularly on $\Z(G)\setminus\{1\}$.
Then Conjecture \ref{conj:Inductive Alperin bound} holds for every $p$-block of $G$ of maximal defect.
\end{lem}

\begin{proof}
Let $\pi$ be the epimorphism $G\twoheadrightarrow S=G/\Z(G)$, and let $Q\in\Syl_p(S)$.
Then $P\in\Syl_p(\pi^{-1}(Q))$ is a Sylow $p$-subgroup of $G$, and $\N_G(P)=\pi^{-1}(\N_S(Q))=P\ti \Z(G)$.
We can identify $\ibr(\N_G(P))$ with $\ibr(\Z(G))$.
For $\la\in\ibr(\Z(G))$, there exists a unique   maximal defect $p$-block $B_\la$ of $G$ such that $\la\in\ibr(\Z(G)\mid\chi)$ for each $\chi\in\ibr(B_\la)$.
By construction, $\la\mapsto B_\la$ is a bijection between $\ibr(\Z(G))$ and the $p$-blocks of $G$ of maximal defect. 
Let $b_\la\in\Bl(\N_G(P))$ be the Brauer correspondent of $B_\la$.
Then $\ibr(b_\la)=\{\la\}$.

If $\la=1_{\Z(G)}$, then the proof of Lemma~\ref{lem:case1} applies, whence we assume that $\la\ne 1_{\Z(G)}$.
Clearly, the irreducible Brauer characters in $B_\la$ have trivial stabilizer under the action of $\Out(G)$.
For every block $B_\la$, we let $\chi_{\la}$ be any Brauer character in $\ibr(B_\la)$.
Then the map $\la\mapsto \chi_\la$ gives an injection $\Omega:\ibr(b_\la)\to\ibr(B_\la)$ which satisfies the required conditions.
\end{proof}	

\begin{pro}\label{prop:iBAW-E6}
	Conjecture \ref{conj:Inductive Alperin bound} holds for every maximal defect 2-block of the universal $2'$-covering groups of the simple groups $\ty{E}_6(q)$ and ${}^2\ty{E}_6(q)$ (with odd $q$).
\end{pro}

\begin{proof}
Let $S$ be the simple group $\ty E_6(\eps q)$ with $\eps=\pm1$ and $2$ not dividing $q$.
Here we let $\ty E_6(\eps q)$ denote $\ty{E}_6(q)$ when $\eps=1$ and ${}^2\ty{E}_6(q)$ when $\eps=-1$.
Assume that $q=q_0^f$ where $q_0$ is an odd prime.
Let $G$ be the universal $2'$-covering groups of $S$.
As $q$ is odd, we have $G=\bG^F$ where $\bG$ is a simple, simply-connected algebraic group of type $\ty E_6$ and $F:\bG\to\bG$ is a Frobenius endomorphism.
Let $E$ be the group generated by field and graph automorphisms so that $E\cong C_{f}\ti C_2$ if $\eps=1$ and  $E\cong C_{2f}$ if $\eps=-1$.

Let $P\in\Syl_2(G)$.
An $\Aut(G)$-equivariant injection $\Omega:\ibr(\N_G(P))\to\ibr(G)$ has been established in the proof of \cite[Thm.~5]{Mal-Nav-Tie23}. We will briefly recall the construction and then verify that $\Omega$ satisfies the required conditions of Proposition \ref{prop:iBAWC-}.
According to \cite[Table~4.5.2]{GLS}, $G$ contains an involution $t$ such that $\bL:=\c_{\bG}(t)$ is a Levi subgroup of $\bG$ such that $L_1:=[\bL,\bL]^F\cong\textup{Spin}_{10}^\eps(q)$.
%Moreover, $t$ is unique up to conjugacy.
By \cite[Prop.~4.3]{Nav-Tie-16}, $\bL$ can be chosen to be $E$-invariant.
Up to conjugacy, we may assume that $t\in P\le L_1$.
Recall that $\bL=\c_{\bG}(s)$ for any $s\in V_1\setminus \Z(G)$.
Moreover, we have that $\N_G(P)=P\ti V_1$ where $V_1=(\Z^\circ(\bL)^F)_{2'}\cong C_{(q-\eps)_{2'}}$, and $L=\textup{O}^{2'}(L)\ti V_1$ for $L:=\bL^F$. From this, we have the correspondences: $\ibr(\N_G(P))\leftrightarrow\irr{V_1}\leftrightarrow$ the set of linear 2-Brauer characters of $L$, whence sometimes we may identify these sets for convenience.
Moreover, there exists a bijection $\la\mapsto b_\la$ between $\irr{V_1}$ and $\Bl(\N_G(P))$ such that $\ibr(b_\la)=\{\la\}$.
Denote by $B_\la:=(b_\la)^G$ the Brauer correspondent of $b_\la$. If $\la=1_{V_1}$, then the proof of Lemma~\ref{lem:case1} applies, whence we assume that $\la\ne 1_{V_1}$ in what follows.

First consider the case $3\nmid(q-\eps)$. Then $G=S$ and we can identify $\bG^F$ with the dual $\bG^{*F^*}$.
Moreover, $\Aut(G)\cong G\rtimes E$.
Let $\chi_\la=\pm\textup{R}_{\bL}^{\bG}(\la)$ where $\la$ is regarded as a linear character of $\bL^F$.
Then $\chi_\la$ is the semisimple character labelled by the semisimple $2'$-element $s\in V_1$ corresponding to $\la$.
By \cite[Prop.~1]{His-Mal01}, the restriction $\chi_\la^\circ$ of $\chi_\la$ to $2'$-elements of $G$ is an irreducible Brauer character.
The injection $\Omega:\ibr(\N_G(P))\to\ibr(G)$ defined in the proof of \cite[Thm.~5]{Mal-Nav-Tie23} satisfies that $\chi_\la^\circ=\Omega(\la)$.
Also, $E_\la$ is cyclic by \cite[Lemma~3]{Mal-Nav-Tie23}. Thus it suffices to show that $\chi_\la\in\irr{B_\la}$.
By construction, $(b_\la)^L$ is the block of $L$ containing $\la$
and thus it follows from \cite[Thm.~4.14]{Bon-Dat-Rou17} that $B_\la$ is associated with $s$ in the sense of \cite[Thm.~9.12]{Cab-Eng04}.
As shown in the proof of \cite[Thm.~5]{Mal-Nav-Tie23}, distinct elements of $V_1\setminus\{1\}$ lie in distinct $G$-conjugacy classes, whence there exists a unique maximal defect 2-block of $G$ associated with $s$ for each $s\in V_1\setminus\{1\}$.
Therefore, $\chi_\la\in\irr{B_\la}$.

In the remaining case $3\mid(q-\eps)$, we use the regular embedding $\bG\hookrightarrow\bH$ constructed in the proof of \cite[Thm.~5]{Mal-Nav-Tie23}. Let $H=\bH^F$ and $\bM=\c_{\bH}(\Z^\circ(\bL))$. Then the group $H\rtimes E$ induces all automorphisms of $G$,
and $P':=P\ti\Z(H)_{2}$ is a Sylow 2-subgroup of $H$. Moreover,  $\N_H(P')=P\ti V$ where $V=(\Z^\circ(\bM)^F)_{2'}$, and
$M=\textup{O}^{2'}(M)\ti V$ for $M=\bM^F$.
For our fixed $\la$, we let $\eta\in\irr{V}$ be an extension of $\la$ to $V$ as in the proof of \cite[Thm.~5]{Mal-Nav-Tie23} such that $E_\la=E_\eta$. Recall that $E_\la$ is cyclic.

We can identify $\bH$ with its dual group $\bH^*$ as in \cite[p.1216]{Mal-Nav-Tie23}.
Then $\chi'_{\eta}=\pm\textup{R}^{\bH}_{\bM}(\eta)$, where $\eta$ is regarded as a linear character of $\bM^F$, is the semisimple character labelled by $s\in V$ corresponding to $\eta$.
As in the proof of \cite[Thm.~5]{Mal-Nav-Tie23}, $\Res^{H}_{G}({\chi'_{\eta}}^\circ)$ is irreducible and
$\Omega(\la)=\Res^{H}_{G}({\chi'_{\eta}}^\circ)$.
Let $A=GE_\la$. Then $A$ satisfies the condition (ii.a) of Proposition~\ref{prop:iBAWC-}. As $E_\la$ is cyclic, the conditions (ii.a) and (ii.b) of Proposition~\ref{prop:iBAWC-} also hold automatically. 
So it suffices to show $\Omega(\la)\in\ibr(B_\la)$ and the condition (ii.d) of Proposition~\ref{prop:iBAWC-}.
By \cite[Lemma~2.3]{Kos-Spa15}, we can establish $\Omega(\la)\in\ibr(B_\la)$ by proving that the block $\bl(\chi'_{\eta})$ is the induced block of $\bl(\eta)$ to $H$ where $\bl(\eta)$ is the block of $\N_H(P')$ containing $\eta$, and this is implied by the above arguments to show that $\chi_\la\in\irr{B_\la}$.
Finally, since $A/G$ is solvable, the proof of \cite[Lemma~7.2]{Fen19} also applies to establish the condition (ii.d) of Proposition~\ref{prop:iBAWC-}, and thus we complete the proof.
\end{proof}	

We are now able to complete the verification of Conjecture \ref{conj:Inductive Alperin bound} for $2$-blocks of maximal defect of finite quasi-simple groups.

\begin{thm}\label{thm:inductive alperin bound for p=2, max defect}
Conjecture \ref{conj:Inductive Alperin bound} holds for all $2$-blocks of maximal defect of every covering group of any non-abelian finite simple group.
\end{thm}
	
\begin{proof}
Let $S$ be a non-abelian finite simple group, and $G$ be a covering group of $S$. 
By \cite[Prop.~2.4]{Mur98} and \cite[Lemma 3.14 and 3.15]{MRR}, it suffices to consider the case that $G$ is the universal $2'$-covering group of $S$.

We first consider the sporadic simple groups. Thanks to Lemma~\ref{lem:case1}, Lemma~\ref{lem:case2}, \cite[Table~5.3]{GLS} and \cite[Cor.]{Kon-05}, it remains to deal with $J_1$, $J_2$, $J_3$, $Suz$ and $HN$.
These groups have been checked to satisfy the inductive BAW condition in \cite{Bre-web}.

The inductive BAW condition was verified in \cite{Mal14} to hold for all simple alternating groups.
Thus we are left to deal with the simple groups $S$ of Lie type. 
Also the defining characteristic case was settled in \cite[Thm.~C]{Spa13I}, whence it suffices to consider groups of Lie type in odd characteristic.

Thanks to Proposition~\ref{prop:iBAW-E6}, we assume that $S$ is not of type $\ty E_6$ or ${}^2\ty E_6$.
By \cite{Fen-Li-Zha-21,Fen-Li-Zha23}, simple groups $\textup{PSL}_n(q)$ and $\textup{PSU}_n(q)$ satisfy the inductive BAW condition.
For the group $G=3.\ty G_2(3)$ or $3.\Omega_7(3)$ we are in the case dealt with in Lemma~\ref{lem:case2}.
Thus all exceptional covering groups (see, e.g., \cite[Table~6.1.3]{GLS}) are settled.
By \cite[Cor.]{Kon-05} and Lemma~\ref{lem:case1}, it remains to consider the following two cases: $S={}^2\ty G_2(q^2)$ or $\textup{PSp}_{2n}(q)$, which have been checked to satisfy the inductive BAW condition by \cite[Prop.~6.4]{Spa13I} and \cite[Thm.~1]{Fen-Mal22} respectively. 
\end{proof}

\section{Proofs of Theorem \ref{thm:Main, p=2} and Theorem \ref{thm:Main, blockwise lower bound p=2, max}}

%\begin{proof}[Proof of Corollary \ref{cor:Main, normal Sylow}]
%Let $P$ be a Sylow $2$-subgroup of $G$. Clearly, if $P$ is normal in $G$ then we have $|\ibr(G)|=|\ibr(\n_G(P))|$. Suppose now that $|\ibr(G)|=|\ibr(\n_G(P))|$. By Theorem \ref{thm:Main Reduction normal Sylow} it suffices to show that Conjecture \ref{conj:Inductive Alperin bound, blockfree} holds for all quasi-simple groups with respect to the prime $2$. This has been verified in \cite{Mal-Nav-Tie23}.
%\end{proof}

Using the reductions contained in Theorem \ref{thm:Reduction} and Theorem \ref{thm:proof of main, normal sylow}, together with the verification of Conjecture \ref{conj:Inductive Alperin bound} for $2$-blocks of maximal defect of finite quasi-simple groups obtained in Theorem \ref{thm:inductive alperin bound for p=2, max defect}, we are finally able to prove Theorem \ref{thm:Main, p=2} and Theorem \ref{thm:Main, blockwise lower bound p=2, max}.

\begin{proof}[Proof of Theorem \ref{thm:Main, p=2}]
Let $P$ be a Sylow $2$-subgroup of $G$. Clearly, if $P$ is normal in $G$ then we have $|\ibr(G)|=|\ibr(\n_G(P))|$. Suppose now that $|\ibr(G)|=|\ibr(\n_G(P))|$. By Theorem \ref{thm:proof of main, normal sylow} it suffices to show that Conjecture \ref{conj:Inductive Alperin bound} holds for all blocks of maximal defect of all quasi-simple groups with respect to the prime $2$. This is done in Theorem \ref{thm:inductive alperin bound for p=2, max defect}.
\end{proof}

\begin{proof}[Proof of Theorem \ref{thm:Main, blockwise lower bound p=2, max}]
By Theorem \ref{thm:Main reduction} and Remark \ref{rem:maximal defect} it suffices to prove Conjecture \ref{conj:Inductive Alperin bound} for the $2$-blocks of maximal defect of the covering groups of the nonabelian finite simple groups, but this is done in Theorem \ref{thm:inductive alperin bound for p=2, max defect}.
\end{proof}

%%%%%%%%%%%%%%%%%%%%%%%%%%%%%%%%%%%%%%%%%%%%%%%%%%%%%%%%%%%%%%%%%%%%%%%%%%%%%%%%%%%%%%%%%%%%%%%%%%%%%%%%%%%%%%%%%%%%%%%%%%%%%%
\section*{Acknowledgements}
Part of this work was done during the visit of the first-named author at FB Mathematik, RPTU Kaiserslautern--Landau, and he would like to thank Gunter Malle and this department for the kind hospitality. The second-named author is grateful to Gabriel Navarro for many insightful conversations about Theorem \ref{thm:Main Reduction normal Sylow} as well as for inspiring the proof of Proposition \ref{pro:glauberman}, and also to Eugenio Giannelli and Mandi Schaeffer Fry for conversations on blocks of finite simple groups.
Part of this project was inspired by a question raised by Lucas Ruhstorfer while the second-named author visited the Bergische Universit\"at Wuppertal in July 2024. He wishes to thank him for conversations on the topic and the whole Darstellungstheorie group for their kind hospitality. The authors wish to thank Gunter Malle and Noelia Rizo for suggestions that helped improve the exposition of this manuscript.

%%%%%%%%%%%%%%%%%%%%%%%%%%%%%%%%%%%%%%%%%%%%%%%%%%%%%%%%%%%%%%%%%%%%%%%%%%%%%%%%%%%%%%%%%%%%%%%%%%%%%%%%%%%%%%%%%%%%%%%%%%%%%%

\bibliographystyle{alpha}

\vspace{1cm}

(Z. Feng) {\sc Shenzhen International Center for Mathematics and Department of Mathematics, Southern University of Science and Technology, Shenzhen 518055, China}

\textit{Email address:} \href{mailto:fengzc@sustech.edu.cn}{fengzc@sustech.edu.cn}

(J. M. Mart\'inez) {\sc{Departament de Matem\`atiques, Universitat de Val\`encia, 46100 Burjassot, Val\`encia, Spain.}}

\textit{Email address:} \href{mailto:josep.m.martinez@uv.es}{josep.m.martinez@uv.es}

(D. Rossi) {\sc{FB Mathematik, RPTU Kaiserslautern--Landau, Postfach 3049, 67653 Kaiserslautern, Germany}}

\textit{Email address:} \href{mailto:damiano.rossi.math@gmail.com}{damiano.rossi.math@gmail.com}

\end{document}